\documentclass[preprint,12pt]{elsarticle}

\usepackage[utf8]{inputenc}
\usepackage[T1]{fontenc}
\usepackage{amssymb,latexsym}
\usepackage{amsmath}
\usepackage{amsthm}
\usepackage{amssymb}
\usepackage{color}
\usepackage[shortlabels]{enumitem}
\usepackage[colorlinks=true,citecolor={blue},draft=false, urlcolor={black}]{hyperref}
\usepackage{cleveref}

\usepackage{ifdraft}
\usepackage{listings}%
\usepackage{esvect} 
\usepackage{todonotes}

\def\RT{\mathrm{RT}}

\def\N{\mathbb N}
\def\A{\mathcal A}
\def\B{\mathcal B}

\def\uu{\mathbf{u}}

\newtheorem{thm}{Theorem}
\newtheorem{theorem}[thm]{Theorem}

\newtheorem{corollary}[thm]{Corollary}
\newtheorem{lemma}[thm]{Lemma}

\newtheorem{proposition}[thm]{Proposition}

\newtheorem{defi}[thm]{Definition}

\crefname{thm}{theorem}{theorems}
\crefname{theorem}{theorem}{theorems}
\crefname{coro}{corollary}{corollaries}
\crefname{example}{example}{examples}
\crefname{lemma}{lemma}{lemmas}
\crefname{claim}{claim}{claims}
\crefname{obs}{observation}{observations}
\crefname{proposition}{proposition}{propositions}
\crefname{prop}{proposition}{propositions}
\crefname{defi}{definition}{definitions}
\crefname{rem}{remark}{remarks}

\newtheorem{remark}[thm]{Remark}
\newtheorem{example}[thm]{Example}

\begin{document}

\begin{frontmatter}

\title{The asymptotic repetition threshold of sequences rich in palindromes}

\author[fn]{Ľubom\'ira Dvo\v r\'akov\'a\corref{ss}}\ead{lubomira.dvorakova@fjfi.cvut.cz}

\author[fi]{Karel  Klouda}

\author[fn]{Edita  Pelantov\'a}

\affiliation[fn]{organization={Department of Mathematics, FNSPE, Czech Technical University in Prague},
            addressline={Trojanova 13}, 
            city={Prague},
            postcode={12000}, 
            country={Czech Republic}}

\affiliation[fi]{organization={Department of Applied Mathematics, FIT, Czech Technical University in Prague},
            addressline={Thákurova 9}, 
            city={Prague},
            postcode={16000}, 
            country={Czech Republic}}

\cortext[ss]{corresponding author}


\begin{abstract} 
The asymptotic critical exponent measures for a sequence the maximum repetition rate of factors of growing length. 
The infimum of asymptotic critical exponents of sequences of a certain class is called the asymptotic repetition threshold of that class.
On the one hand, if we consider the class of all $d$-ary sequences with $d\geq 2$, then the asymptotic repetition threshold is equal to one, independently of the alphabet size.
On the other hand, for the class of episturmian sequences, the repetition threshold depends on the alphabet size. 
We focus on rich sequences, i.e., sequences whose factors contain the maximum possible number of distinct palindromes. 
The class of episturmian sequences forms a subclass of rich sequences.
We prove that the asymptotic repetition threshold for the class of rich recurrent $d$-ary sequences, with $d\geq 2$, is equal to two, independently of the alphabet size.

\end{abstract}


\begin{keyword}
repetition threshold \sep asymptotic critical exponent \sep sequences rich in palindromes \sep bispecial factors 
\MSC 68R15
\end{keyword}

\end{frontmatter}




\section{Introduction}\label{sec:Introduction}

In this paper, we aim to provide a survey on the repetition threshold and the asymptotic repetition threshold of several classes of sequences (infinite words) and prove a new result on the asymptotic repetition threshold of rich sequences. 

Let us recall first the essential notions related to repetitions. For a positive integer $n$, the $n$-th power of a non-empty word $u$, denoted $u^n$, is a concatenation of $n$ copies of $u$.
In other words, $u^n$ denotes the prefix of length $|u|n$ of the infinite periodic sequence $uuu\cdots = u^\omega$, where $|u|$ is the length of $u$. 
As an example, consider a square $bonbon=(bon)^2$ and a cube $ninini$\footnote{Ninini rodiče in Czech means Nina's parents.} $=(ni)^3$.
Besides integer powers we can define also rational powers: any prefix of length $k$ of $u^\omega$ can be written as $u^e$, where $e=\frac{k}{|u|}$. 
For instance, a Czech sentence  $stel{\_}postel{\_}poste$\footnote{Translation: Stel postel posté! = Make your bed for a hundredth time!}  can be written in this formalism  as  $(stel{\_}po)^{17/7}$.  

\medskip

The {\em critical exponent}  $E(\uu )$ of a sequence $\uu=u_0 u_1 u_2 \cdots$ is defined as
$$E(\uu) =\sup\{e \in \mathbb{Q}: \  u ^e \  \text{is a factor of   } \uu  \  \text{for a non-empty word} \  u\}\,.
$$
In this paper we concentrate on  the {\em asymptotic critical exponent} $E^*(\uu)$, introduced by Cassaigne~\citep{Ca2008} under the name asymptotic index, which 
 is defined to be $+\infty$,  if $E(\uu) = +\infty$, otherwise 
 $$E^*(\uu) =\lim_{n\to \infty}\sup\{e \in \mathbb{Q}: \  u ^e \  \text{is a factor of  } \uu  \  \text{for some }  u \ \text{of length} \geq  n  \}\,.$$
It is readily seen that $E^*(\uu)\leq E(\uu)$.
\medskip

An important parameter of any class $C$ of sequences is the infimum of critical exponents, resp. of asymptotic critical exponents, of sequences from this class.
We call this parameter {\em repetition threshold} for $C$ 
$$\mathrm{RT}(C) = \inf\{E(\uu): \uu \ \text{\ belongs to the class\ } C\}\,,$$
resp. {\em asymptotic repetition threshold} for $C$
$$\mathrm{RT}^*(C) = \inf\{E^*(\uu): \uu \ \text{\ belongs to the class\ } C\}\,.$$
Obviously, $\mathrm{RT}^*(C)\leq \mathrm{RT}(C)$.

For the class of all $d$-ary sequences the repetition threshold and the asymptotic repetition threshold is usually denoted $\mathrm{RT}(d)$ and $\mathrm{RT}^*(d)$, respectively. Each binary sequence contains squares and already Axel Thue showed in 1912~\citep{Thue1912} that the Thue-Morse sequence, the fixed point of the morphism $0 \mapsto 01$ and $ 1\mapsto 10$, has the critical exponent equal to two. Hence, $\mathrm{RT}(2) = 2$ and the threshold is reached for the Thue-Morse sequence. Dejean \citep{Dej72} showed that $\mathrm{RT}(3)=7/4$ and conjectured the remaining values $\mathrm{RT}(4)=7/5$ (proved by Pansiot \citep{Pan84c}) and $\RT(d)=1+\frac{1}{d-1}$ for $d\geq 5$ (proved by efforts of many authors \citep{Mou92,Car07,CuRa11,Rao11}). 
In contrast to the repetition threshold, for the asymptotic repetition threshold we have $\RT^*(d) =1 \ \text{for all $d\geq 2$}\,,$ see~\citep{Ca2008}.

Let us briefly survey what is known about $\RT(C)$ for different classes $C$.
The most studied aperiodic binary sequences are Sturmian sequences~\citep{MiPi1992, MoCu07}. The most prominent example is the Fibonacci sequence, which is the fixed point of the morphism $0 \mapsto 01$ and $1\mapsto  0$.
The Fibonacci sequence holds a privileged position among all Sturmian sequences, see~\citep{Ca2008}. In particular, in the class of Sturmian sequences, the Fibonacci sequence has the minimal value of both the critical and the asymptotic critical exponent. As shown by Carpi and de Luca~\citep{CaDeLu00} for the repetition threshold and Vandeth~\citep{Van2000} for the asymptotic version, in the class of Sturmian sequences $C^{(s)}$, it holds $\RT(C^{(s)})=\RT^*(C^{(s)})=2+\frac{1+\sqrt{5}}{2}$. 

Sturmian sequences are episturmian, balanced and rich. The (asymptotic) repetition threshold of episturmian and balanced sequences has been recently studied.  
For the class $C^{(e)}_d$ of $d$-ary episturmian sequences with $d\geq 2$, we have shown~\cite{DvPe2024} that 
\begin{equation}\label{eq:ARsequence}\RT(C^{(e)}_d)= \RT^*(C^{(e)}_d) =  2+\frac{1}{t_d-1}, 
\end{equation}
where $t_d>1$ is the only positive root of the polynomial $x^d-x^{d-1}-\dots -x-1$.

Rampersad, Shallit and Vandomme~\citep{RSV19} suggested to study the repetition threshold for the class $C^{(b)}_d$ of $d$-ary balanced sequences. The following results have been proved so far:
\begin{itemize}
    \item $\RT(C^{(b)}_2)=2+\frac{1+\sqrt{5}}{2}$ \citep{CaDeLu00};
    \item $\RT(C^{(b)}_3)=2+\frac{\sqrt{2}}{2}$ and $\RT(C^{(b)}_4)=1+\frac{1+\sqrt{5}}{4}$ \citep{RSV19};
 \item $\RT(C^{(b)}_d)=1+\frac{1}{d-3}$ for $5 \leq d\leq 10$ \citep{BaSh19, DolceDP2023};  
    \item $\RT(C^{(b)}_d)=1+\frac{1}{d-2}$ for $d=11$ and all even numbers $d \geq 12$ \citep{DvOpPeSh2022}.
    \end{itemize}
It remains an open problem to prove the conjecture $\RT(C^{(b)}_d)=1+\frac{1}{d-2}$ also for all odd numbers $d\geq 13$.

For the asymptotic repetition threshold in this class, the situation is more interesting than for general $d$-ary sequences, see~\citep{DOP2022} and \citep{DvPe2022}:   
\begin{itemize}
    \item $\RT(C^{(b)}_d)= \RT^*(C^{(b)}_d)$ for $2\leq d \leq 5$;
 \item $\RT(C^{(b)}_d)> \RT^*(C^{(b)}_d)$ for $d \geq 6$; 
    \item $1+ \frac{1}{2^{d-2}}< \RT^*(C^{(b)}_d)< 1+ \frac{\tau^3}{2^{d-3}}$, where $\tau = \frac{{1}+\sqrt{5}}{2}$;
    \item the precise value of $\RT^*(C^{(b)}_d)$ is known only for $d \leq 10$. 
\end{itemize}

Here, we are interested in the class of rich sequences. Their introduction was initiated by observation of Justin et al.~\cite{DrJuPi2001} that the number of distinct palindromic factors (including the empty word) contained in a word of length $n$ is smaller than or equal to $n+1$. 

If a word contains the maximum number of palindromes, it is called {\em rich}.
We keep here the terminology introduced by Glen et al. in 2009~\cite{GlJuWi2009}.
The number of rich words of length $n$ over an alphabet of size $d$ is not known. Upper and lower bounds may be found in~\cite{Guo2016, Rukavicka2017}. 
A sequence is called {\em rich} if all its factors are rich. Examples of rich sequences include episturmian sequences, complementary-symmetric Rote sequences \citep{BBLV2011}, sequences coding symmetric $d$-interval exchange \citep{BaMaPe2007}, the period doubling sequence etc.
Rich sequences may be characterized by various ways:  using 
either  complete return words~\cite{GlJuWi2009} or using a relation between factor and palindromic complexity~\cite{BuLuGlZa2009} or  the extensions of bispecial factors~\cite{BaPeSt2010}. 

The repetition threshold of rich sequences has been so far studied only for binary and ternary sequences. The value of the repetition threshold for the binary alphabet was conjectured in \citep{BaSh19}, and proved by Currie, Mol and Rampersad \citep{CuMoRa2020}.
If $C^{(r)}_2$ is the set of all binary rich sequences, then 
$$ \RT(C^{(r)}_2) =2+\frac{\sqrt{2}}{2} \approx 2.707\,.$$
Moreover, Baranwal and Shallit~\citep{BaSh19} found by backtracking a lower bound $\RT(C^{(r)}_3)\geq  9/4$ for the class $C^{(r)}_3$ of ternary rich sequences. Recently, Currie, Mol and Peltom\"{a}ki\footnote{Private communication from July 2024.} have shown that 
$\RT(C^{(r)}_3)=1 + \frac{1}{3-\mu}\approx 2.259$, where $\mu$ is the unique real root of the polynomial $x^3-2x^2-1$.

In this paper, we focus on the class of sequences 
$$C^{(r)}_d = \{\uu: \uu \text{ \ is a recurrent rich $d$-ary sequence}\}. $$
Known results provide the following lower and upper bound
$$2\leq \RT^*(C^{(r)}_d)  \leq \RT(C^{(r)}_d) \leq 2+ \tfrac{1}{t_d-1},$$ where $t_d$ is as in \eqref{eq:ARsequence}.
The inequality $2\leq \RT^*(C^{(r)}_d)$ was proven in~\citep{PStarosta2013}. The inequality $\RT(C^{(r)}_d) \leq 2+ \tfrac{1}{t_d-1}$ is a consequence of the fact that $d$-ary episturmian sequences belong to the class $C^{(r)}_d$. 
Note that the upper bound  $2+\tfrac{1}{t_d-1} > 3$ for every $d \in \mathbb{N}, d\geq  2$. 

Our new result states that the announced lower  bound $2$ for the asymptotic repetition threshold is attained for every alphabet size.
\begin{theorem}\label{thm:main} Let $d \in \N, d \geq 2$. Then 
    $$\RT^*(C^{(r)}_d) =\inf\{E^*(\uu): \uu \text{\ is a recurrent rich $d$-ary sequence}\}\,  = 2.$$
\end{theorem}
Our proof is based on the study of the fixed point $\uu_d$ of the morphism
\begin{equation}\label{eq:morphismFi}\begin{array}{rrcl} 
\varphi_d: 
&  0 & \to &   01 \\
&  1 & \to &   02 \\
&   2 & \to &   03 \\
& \vdots~ && ~\vdots \\
&   d-2 & \to &   0(d-1) \\
&   d-1 & \to &   0(d-1)(d-1) 
\end{array}\end{equation}
and subsequent application of a suitable projection of $\uu_d$ to the binary alphabet.

Let us note that the morphism $\varphi_3$ was introduced by Baranwal and Shallit~\citep{BaSh19}. By a~suitable projection of $\uu_3$ to the binary alphabet, they obtained a rich sequence with the critical exponent equal to $2+\frac{\sqrt{2}}{2}$, i.e., the upper bound on $\RT(C^{(r)}_2)$.
In the case $d=2$,  the fixed point $\uu_2$ of $\varphi_2$ equals $0{\bf f}$, where ${\bf f}$ is the Fibonacci sequence fixed by the morphism $0\mapsto 1$ and $1\mapsto 10$. This case is well understood. Consequently, when studying properties of $\uu_d$ in the sequel, we consider only $d\geq 3$. 

\medskip

After Introduction and Preliminaries, the remaining text is devoted to the proof of Theorem~\ref{thm:main}. It consists of the following steps: 
In Section~\ref{sec:BSfactors} we study the structure of bispecial factors of $\uu_d$ and show that   $\uu_d$  is rich in palindromes, see  Theorem~\ref{coro:richness}.  Section~\ref{sec:return} focuses on  description of return words to bispecial factors in $\uu_d$, which play an essential role in computation of the asymptotic critical exponent. In Section~\ref{sec:asymptotic}, we determine the asymptotic critical exponent of $\uu_d$,  see Theorem~\ref{thm:HodnotaE*}.   
Theorem \ref{thm:main}  is finally proven in Section~\ref{sec:konecne}.

\section{Preliminaries}
An \textit{alphabet} is a finite set of symbols, called \textit{letters}. Throughout the text, we work with the alphabet $\mathcal{A} = \{0,1,\ldots, d-1\}$. 
A \textit{word} $u$ over $\mathcal A$ of \textit{length} $n$ is a finite string $u = u_0 u_1 \cdots u_{n-1}$, where $u_j\in\mathcal A$ for all $j \in \{0,1,\dots, n-1\}$. The length of $u$ is denoted $|u|$ and $|u|_{i}$ denotes the number of occurrences of the letter ${i}\in\mathcal A$ in the word $u$. The \textit{Parikh vector} $ \vec{u} \in \N^{d}$ is the vector defined as ${\vec u } = (|u|_{0}, |u|_{1}, \dots, |u|_{d-1})^{ T}$. The set of all finite words over $\A$ is denoted $\A^*$. The set $\A^*$ equipped with concatenation as the operation forms a monoid with the \textit{empty word} $\varepsilon$ as the neutral element. Consider $u, p, s, v \in \A^*$ such that $u=pvs$, then the word $p$ is called a \textit{prefix}, the word $s$ a \textit{suffix} and the word $v$ a \textit{factor} of $u$. Moreover, $us^{-1}$ denotes the word $u$ without the suffix $s$, i.e., $us^{-1}=pv$.

A~\textit{sequence} $\uu$ over $\A$ is an infinite string $\uu = u_0 u_1 u_2 \cdots$ of letters $u_j \in \A$ for all $j \in \N$. A \textit{word} $w$ over $\mathcal A$ is called a~\textit{factor} of the sequence $\uu = u_0 u_1 u_2 \cdots$ if there exists $j \in \mathbb N$ such that $w = u_j u_{j+1} u_{j+2} \cdots u_{j+|w|-1}$. The integer $j$ is called an \textit{occurrence} of the factor $w$ in the sequence $\uu$. If $j=0$, then $w$ is a \textit{prefix} of $\uu$.

The \textit{language} $\mathcal{L}(\uu)$ of a sequence $\uu$ is the set of factors occurring in $\uu$.
The \textit{factor complexity} of a sequence $\uu$ is a mapping ${\mathcal C}:\mathbb N \to \mathbb N$, where $${\mathcal C}(n)=\#\{w \in {\mathcal L}(\uu) \ : \ |w|=n\}\,.$$ 
The language $\mathcal{L}(\uu)$ is called \textit{closed under reversal} if for each factor $w=w_0w_1\cdots w_{n-1}$, its \textit{mirror image} $\overline{w}=w_{n-1}\cdots w_1 w_0$ is also a factor of $\uu$.
A~factor $w$ of a sequence $\uu$ is \textit{left special} if $iw, jw \in \mathcal{L}(\uu)$ for at least two distinct letters ${i, j} \in \A$. A \textit{right special} factor is defined analogously. A factor is called \textit{bispecial} if it is both left and right special. The set of left extensions is denoted $\mathrm{Lext}(w)$, i.e., $\mathrm{Lext}(w)=\{iw \in {\mathcal L}(\uu)\ : \ i \in {\mathcal A}\}$. Similarly, $\mathrm{Rext}(w)=\{wi \in {\mathcal L}(\uu) \ :\ i \in {\mathcal A}\}$. The set of both-sided extensions is denoted $\mathrm{Bext}(w)=\{iwj \in {\mathcal L}(\uu)\ :\ i,j \in {\mathcal A}\}$. The \textit{bilateral order} $\mathrm{b}(w)$ of $w\in {\mathcal L}(\uu)$ is defined 
$$\mathrm{b}(w)=\#\mathrm{Bext}(w)-\#\mathrm{Lext}(w)-\#\mathrm{Rext}(w)+1\,.$$
Factors that are not bispecial have evidently the bilateral order equal to zero. 
A bispecial factor $w \in {\mathcal L}(\uu)$ is called \textit{ordinary} if $\mathrm{b}(w)=0$.

A word $w$ is a \textit{palindrome} if $w$ is equal to its mirror image, i.e., $w=\overline{w}$.
If $w$ is a~palindromic factor of a sequence $\uu$ over $\mathcal A$, then the set of palindromic extensions of $w$ is $\mathrm{Pext}(w)=\{iwi \in {\mathcal L}(\uu)\ : \ i \in {\mathcal A}\}$.

A sequence $\uu$ is \textit{recurrent} if each factor of $\uu$ has infinitely many occurrences in $\uu$. Moreover, a recurrent sequence $\uu$ is \textit{uniformly recurrent} if the distances between consecutive occurrences of each factor in $\uu$ are bounded. If a uniformly recurrent sequence $\uu$ contains infinitely many palindromic factors, then its language ${\mathcal L}(\uu)$ is closed under reversal. 
A~sequence $\uu$ is \textit{eventually periodic} if there exist words $w \in \A^*$ and $v \in \A^* \setminus \{\varepsilon\}$ such that $\uu$ can be written as $\uu = wvvv \cdots = wv^\omega$. If $\uu$ is not eventually periodic, $\uu$ is called \textit{aperiodic}.

Consider a factor $w$ of a recurrent sequence $\uu = u_0 u_1 u_2 \cdots$. Let $j < \ell$ be two consecutive occurrences of $w$ in $\uu$. Then the word $u=u_j u_{j+1} \cdots u_{\ell-1}$ is a \textit{return word} to $w$ in $\uu$ and $uw$ is a \textit{complete return word} to $w$ in $\uu$.

A \textit{morphism} is a map $\psi: \A^* \to \B^*$ such that $\psi(uv) = \psi(u)\psi(v)$ for all words $u, v \in \A^*$.
The morphism $\psi$ can be naturally extended to a sequence $\uu=u_0 u_1 u_2\cdots$ over $\A$ by setting
$\psi(\uu) = \psi(u_0) \psi(u_1) \psi(u_2) \cdots\,$.
Consider a factor $w$ of $\psi({\uu})$. We say that $(w_1, w_2)$ is a \emph{synchronization point} of $w$ if $w=w_1w_2$ and for all $p,s \in {\mathcal L}(\psi({\uu}))$ and $v \in {\mathcal L}({\uu})$ such that $\psi(v)=pws$ there exists a factorization $v=v_1v_2$ of $v$ with $\psi(v_1)=pw_1$ and $\psi(v_2)=w_2s$. We denote the synchronization point by $w_1\bullet w_2$.

A \textit{fixed point} of a morphism $\psi:  \A^* \to  \A^*$ is a sequence $\uu$ such that $\psi(\uu) = \uu$.
We associate to a morphism $\psi: \A^* \to  \A^*$ the \textit{incidence matrix} $M_\psi$ defined for each $i,j \in \{0,1,\dots, d-1\}$ as $(M_\psi)_{ij}=|\psi(j)|_{i}$. 
A morphism $\psi$ is \textit{primitive} if the matrix $M_\psi$ is primitive, i.e., there exists $k\in \mathbb N$ such that $M_\psi^k$ is a positive matrix.   

By definition, we have for each $u \in \A^*$ the following relation for the Parikh vectors $\vec{\psi}(u)=M_\psi\vec{u}$.

 \medskip 
Let ${\uu}$ be a sequence over $\mathcal A$. Then the \textit{uniform frequency} of the letter ${i}\in \mathcal A$ is equal to $f_i$ if for any sequence $(w_{n})$ of factors of ${\uu}$ with increasing lengths 
$$f_i=\lim_{n\to \infty}\frac{|w_{n}|_{i}}{|w_{n}|}\,.$$
It is known that fixed points of primitive morphisms have uniform letter frequencies~\cite{Quef87}.

Last, but not least, let us summarize basic properties of the morphism $\varphi_d$, introduced in \eqref{eq:morphismFi}, and its fixed point $\uu_d$:

\begin{enumerate}
\item The morphism $\varphi_d$ is primitive and injective.
\item If $w \in \A^*$ is a palindrome, then $\varphi_d(w)0$ is a palindrome. Hence, $\uu_d$ contains infinitely many palindromic factors.
\item The sequence $\uu_d$ is uniformly recurrent since $\uu_d$ is a fixed point of a primitive morphism. 
\item The language of $\uu_d$ is closed under reversal. This follows putting Items 2 and 3 together. 
\end{enumerate}


\section{Richness of $\uu_d$ and its bispecial factors }\label{sec:BSfactors}
The aim of this section is to show that each bispecial factor of the fixed point $\uu_d$ of $\varphi_d$, where $d \geq 3$, has the bilateral order equal to zero.
Richness of $\uu_d$ is then a consequence of the following theorem. 
\begin{theorem}[Corollary 5.10 \citep{BaPeSt2010}]\label{thm:richness_ordinaryBS} 
Let $\uu$ be a sequence with the language closed under reversal.
If all bispecial factors are ordinary, then $\uu$ is rich.
\end{theorem}

In the sequel, we fix $d\geq 3$, ${\mathcal A}=\{0,1,\dots, d-1\}$ and we abbreviate $\varphi = \varphi_d$ and $\uu = \uu_d$. 

First, we focus on both-sided extensions of bispecial factors.
It is readily seen that  
\begin{equation}\label{eq:list2factors}
(d-1)(d-1) \quad\text{and } \quad 0i, \  i0 \quad \text{ for }   i=1,2, \ldots, d-1    
\end{equation}
is the list of all factors of length 2 in $\uu$, in particular, it is equal to $\mathrm{Bext}(\varepsilon)$.

All these $2$-letter factors have synchronization points as there is always a synchronization point before $0$, i.e.,  using the notation from above we have $\bullet 0i$ and $i \bullet 0$,  and $(d-1)(d-1)$ occurs only as the suffix of $\varphi(d-1)$, i.e., $(d-1)(d-1)\bullet$.
As for the factors of length one, the factors $\bullet0, 1\bullet, \ldots, (d-2)\bullet$ all have synchronization points. Only the factor $d-1$ is without a~synchronization point as it can appear in both $\varphi(d-1) = 0(d-1)(d-1)$ and $\varphi(d-2) = 0(d-1)$.

Synchronization points can help us find preimages (under the injective morphism~$\varphi$) of factors of $\mathbf{u}$. For instance, we will repetitively use this fact: if a word $w$ begins and ends with $0$, there is a unique $w'$ such that $w = \varphi(w')0$.

\begin{lemma}\label{lem:Neni} 
Let $i,j,k,l \in \mathcal{A}$ be such that $1)$ $i\neq j$ and  $k\neq l$  and $2)$ $\max\{i,j\} \leq \min\{k,l\}$. Then  no  word $ w \in \mathcal{A}^*$ satisfies that both $iwk$ and $jwl$ belong to $\mathcal{L}(\uu)$. 
\end{lemma}

\begin{proof} 
Let us assume for contradiction that such a word $w$ exists. We choose the shortest one among all such words.   

First assume $w =\varepsilon$. The conditions $1)$ and $2)$ imply that $i,j \leq d-2$ and $1\leq  k,l$. Hence $ik$ or $jl$ does not belong to the list~\eqref{eq:list2factors} --  a contradiction.

Now consider a nonempty word $w=w_0w_1\cdots w_{n-1}$ with  $n \in \mathbb{N}$. The assumption $iw, jw \in \mathcal{L}(\uu)$ and the list \eqref{eq:list2factors} give $w_0 = 0$ and  $i,j\geq 1$. By the same reason, $w_{n-1} = 0$. Hence, by definition of $\varphi$ in~\eqref{eq:morphismFi}, we have $w = \varphi(w')0$ for some $w' \in \mathcal{L}(\uu)$. 
As $1\leq i,j \leq d-2$, necessarily $0iwk = \varphi(i'w')0k \in \mathcal{L}(\uu)$, where $i'= i-1$, and $0jwl = \varphi(j'w')0l \in \mathcal{L}(\uu)$, where $j'=j-1$. Consequently, there exist distinct letters $k',l'\in \{k-1, l-1, d-1\}$ such that $i'w'k'$ and $j'w'l'$ belong to  $\mathcal{L}(\uu)$. We have thus found a shorter word $w'$ and letters $i',j', k',l'$ satisfying the assumptions of the lemma -- a contradiction.   
\end{proof}

To simplify notation we put
\begin{equation}\label{eq:DefF} 
    F_0=\varepsilon \quad \text{and} \quad \ F_i=\varphi(F_{i-1})0\,
    \quad \text{for}\quad  i=1,2,\ldots, d-1\,.
\end{equation}
Note that $F_i$ is a palindrome for all $i$ by the second property of $\varphi=\varphi_d$ in Preliminaries.

\begin{proposition}\label{prop:3extensions}  
Let $w$ be a bispecial factor in $\uu$ having at least two distinct letters smaller than $d-1$ as its left extensions. 
Denote $i =\min \{a\in \mathcal{A}: aw \in \mathcal{L}(\uu)\}$.  
Then $w = F_i$ and all both-sided extensions of $w$  are
$(d-1)w(d-1), iwk, kwi$, where $ k =i+1, i+2, \ldots, d-1$. 

In particular, the bilateral order of $w=F_i$  equals $0$ and $i\leq d-3$.    
    
\end{proposition}
\begin{proof} 
We proceed  by induction on $i$. 

As $i=0$ and a letter $j,\ 1\leq j\leq d-2,$ are left extensions of $w$, the list \eqref{eq:list2factors} forces $w =\varepsilon =F_0$. The list of  both-sided extensions  of $w$ coincides with $\eqref{eq:list2factors}$, which is in agreement with the statement of the proposition.  

Now assume $i \geq 1$. As $0$ does not belong to the left extensions of $w$, but $i,\ 1\leq i\leq d-2,$\ does,  $w= w_0w_1\cdots w_{n-1}$ with $w_0 =0$.  Let $j$ be a left extension of $w$, where   $i<j<d-1$. Obviously $0iw$  and $0jw$  belong to $\mathcal{L}(\uu)$.  Since  $w$ has at least  two right extensions, $w_{n-1} \in \{0,d-1\}$. 
\begin{description}
    \item[Case $w_{n-1} = 0$.]  \ \\
    There exists $w'$ such that $w=\varphi(w')0$. 
    
    Since $0iw=\varphi\bigl((i-1)w'\bigr)0$ and  $0jw  = \varphi\bigl((j-1)w'\bigr)0$,  the letters $i-1$ and $j-1$ are left extension of $w'$. In other words, $w'$ is a left special factor and $i-1$ is its minimal left extension. 
    
    As $wk =  \varphi(w')0k$  and $wl =  \varphi(w')0l$  are factors of $\uu$ for some letters $k,l, k \not =l$, the factor $w'$ is right special as well. By induction hypothesis, $w' = F_{i-1}$ and thus $w = \varphi(F_{i-1})0 = F_i $.  Applying $\varphi$ to the both-sided extensions of $F_{i-1}$, we easily deduce that the both-sided extensions of $w$  coincides with the list stated in the proposition.   
    
    \item[Case $w_{n-1} = d-1$.] \  \\
    The factor $w$ can be prolonged to the right  only by two letters, say  $k, l$, where  $\{k,l\} = \{0,d-1\}$.  Then   the penultimate letter of $w$ is $w_{n-2} = 0$ and $w=u0(d-1)$, where $u$ is either the empty word or $u$ has a prefix $0$.  It implies that  $u=\varphi(u')$ for some $u'\in {\mathcal L}(\uu)$. 
    
    The fact that 

    \[
        \mathcal{L}(\uu) \ni 0iwk = 0iu0(d-1)k = \varphi\bigl((i-1)u'\bigr)0(d-1)k
    \]
    and 
    \[
        \mathcal{L}(\uu) \ni  0jwl = 0ju0(d-1)l = \varphi\bigl((j-1)u'\bigr)0(d-1)l
    \]
    
    implies that  $(i-1)u'k'$ and $(j-1)u'l'$  belong to the language of $\uu$ with $\{k', l'\} = \{d-2, d-1\}$  and $i-1, j-1\leq d-3$. Existence of $u'$ contradicts Lemma \ref{lem:Neni}.  Thus the case $w_{n-1} = d-1$ cannot occur. 
 \end{description}
 
\end{proof}
Since the language is closed under reversal, an analogous statement holds for right special extensions, too.
\begin{corollary} \label{coro:d-1}
The letter $d-1$ belongs to the left and to the right extensions of each bispecial factor of $\uu$ and the factors $F_0, F_1, \ldots, F_{d-3}$ are the only bispecial factors with more than two left extensions or more than two right extensions.
\end{corollary} 
To describe the structure of bispecial factors with exactly two left and two right extensions, we use the method by Klouda~\cite{K2012}.
This method does not work directly with bispecial factors, but with bispecial triplets $((a,b), w, (g,h))$, where $a, b, g, h$ are letters\footnote{In general, they can be words longer than one, but in the case of $\varphi$ it suffices to consider letters. We also use the fact that our letters are ordered.} and $a < b,\ g < h$ and at least one of $awg$ and $awh$ and at least one of $bwg$ and $bwh$ are factors. 
Clearly, each bispecial factor is a middle element of a bispecial triplet. We say that the bispecial factor is \textit{associated} with the bispecial triplet. 
However, if there are more than two left or two right extensions, there is more than one such bispecial triplet.

The method~\cite{K2012} says that there is a mapping that maps one bispecial triplet to another and that all bispecial triplets are obtained by repetitive application of the mapping to a finite set of \textit{initial} bispecial triplets.
This mapping is defined by two directed labeled graphs of left and right extensions.
For the morphism $\varphi$ the construction of the graphs is straightforward.

\noindent For left extensions, the set of vertices is $\{ (a, b) : a, b \ \text{ are letters}, \ a < b \}$ and the edges are:
\begin{equation}\label{eq:left}
\begin{array}{rcll}
    (a,b) &\overset{\varepsilon}{\longrightarrow}  & (a + 1, b + 1) & \text{if } a < b < d - 1\,; \\
    (a, d-1)& \overset{\varepsilon}{\longrightarrow}  & (a + 1, d - 1) & \text{if } a < d - 2\,; \\
   (d-2, d-1)& \overset{d-1}{\longrightarrow}  & (0, d - 1)\,, &
\end{array}
\end{equation}
i.e., the labels of the edges are given by the longest common suffix of $\varphi(a)$ and $\varphi(b)$, we also use notation ${\rm lcs\,}\{\varphi(a), \varphi(b)\}$.

\noindent For right extensions, the set of vertices is $\{ (g, h) : g,h \ \text{ are letters}, \ g < h \}$ and the edges are:
\begin{equation}\label{eq:right}
\begin{array}{rcll}
    (g,h) & \overset{0}{\longrightarrow}  & (g + 1, h + 1) & \text{if } g < h < d - 1\,; \\
    (g, d-1) & \overset{0}{\longrightarrow}  & (g + 1, d - 1) & \text{if } g < d - 2\,; \\
   (d-2, d-1) & \overset{0(d-1)}{\longrightarrow}  & (0, d - 1)\,, &
\end{array}
\end{equation}
i.e., the labels of the edges are given by the longest common prefix of $\varphi(g)$ and $\varphi(h)$, also denoted ${\rm lcp\,}\{\varphi(g), \varphi(h)\}$.

The mapping, in the paper called $f$\textit{-image}, is defined as follows: The $f$-image of a bispecial triplet $((a,b), w, (g,h))$ is the bispecial triplet $((a',b'), u_1\varphi(w)u_2, (g',h'))$ if the edges starting in $(a,b)$ and $(g,h)$ end in $(a',b')$ and $(g',h')$ with labels $u_1$ and $u_2$, respectively. 
\begin{example}
    For $d = 5$, $((0,4), \varepsilon, (2,4))$ is a bispecial triplet as $02, 04$ and $44$ are factors.
    Its first three $f$-images are
    \begin{align*}
        f\text{-image:} \quad & ((1,4), 0, (3,4)) \\
        f^2\text{-image:} \quad & ((2,4), \varphi(0)04, (0,4)) \\
        f^3\text{-image:} \quad & ((3,4), \varphi^2(0)\varphi(0)\varphi(4)0, (1,4))\,.
    \end{align*}
\end{example}
The set of initial bispecial triplets corresponds to, as proved in~\cite{K2012}, the set of bispecial factors without a~synchronization point.
Hence, in our case, the initial triplets are: \\
$((0, d-1), d-1, (0, d-1))$ and $((a,b), \varepsilon, (g,h))$ for admissible letters $a < b$ and $g < h$, given in~\eqref{eq:list2factors}.

Let us observe a useful property of bispecial triplets.
\begin{lemma}\label{lem:vzdy_d-1} 
Let $((a,b), w, (g,h))$ be a bispecial triplet of $\uu$.
\begin{itemize}
\item If $b<d-1$, then $((a,d-1), w, (g,h))$ is a bispecial triplet and the bispecial factors associated with the $f$-images of $((a,b), w, (g,h))$ and of $((a,d-1), w, (g,h))$ coincide.
\item If $h<d-1$, then $((a,b), w, (g,d-1))$ is a bispecial triplet and the bispecial factors associated with the $f$-images of $((a,b), w, (g,h))$ and of $((a,b), w, (g,d-1))$ coincide.
\end{itemize}
\end{lemma}
\begin{proof}
Let us prove the first statement, the second one is analogous. 
If $a$ and $b$, such that $a<b<d-1$, are left extensions of $w$, then by Proposition~\ref{prop:3extensions}, $w=F_i$ for some $i=0,1,\ldots, d-3$ and the both-sided extensions of $w$ are: $(d-1)w(d-1), iwk, kwi$, where $ k =i+1, i+2, \ldots, d-1$. Hence $((a,d-1), w, (g,h))$ is a bispecial triplet, too, and by definition of $f$,  the $f$-image of $((a,b), w, (g,h))$ equals $((a+1, b+1), \varphi(w)u_2, (g',h'))$ and the $f$-image of $((a,d-1), w, (g,h))$ equals $((a+1,d-1), \varphi(w)u_2, (g',h'))$ for $u_2={\rm lcp\,}\{\varphi(g),\varphi(h)\}$. Consequently, the associated bispecial factor $\varphi(w)u_2$ is the same in both cases. 
\end{proof}

By Lemma~\ref{lem:vzdy_d-1}, when identifying all bispecial factors, it suffices to consider bispecial triplets of the form $((a, d-1), w, (g, d-1))$. We can thus significantly simplify the notation for bispecial triplets and $f$-images without losing any information.

Let us set
\[
    \mathcal{T} = \{(a,w,b): w \text{ bispecial factor of } \uu, \   a,b \in \mathcal{A},\ a,b < d-1,\ aw, wb \in \mathcal{L}(\uu)\}\,.
\]
Obviously, every bispecial factor $w$ occurs as a middle element of a triplet $T$ from  $\mathcal{T}$. 
We say that $w$ is a bispecial factor associated with the triplet $T$.

Analogously, we can simplify the definition of $f$-image to be applicable on triplets from $\mathcal{T}$:
\begin{defi}\label{def:Zobrazeni_f} 
For $(a,w,b)\in  \mathcal{T}$ set $f(a,w,b) = (a',w',b')$, where 
$ a'=a+1 \!\!\!\mod (d-1)$,  $b' = b+1 \!\!\!\mod (d-1)$ and 
\[
  w' =\left\{\begin{array}{ll} \phantom{(d-1)}\varphi(w)0& \text{ if } a<d-2 \text{ and } b<d-2\,;\\
  \phantom{(d-1)}\varphi(w)0(d-1)\phantom{(d-1)}& \text{ if } a<d-2 \text{ and } b=d-2\,;\\

  (d-1)\varphi(w)0& \text{ if } a=d-2 \text{ and } b<d-2\,;\\

  (d-1)\varphi(w)0(d-1)& \text{ if } a=d-2 \text{ and } b=d-2\,.\\
  \end{array} \right.
\]
We call the bispecial factor $w'$ the \emph{offspring} of $w$. 
\end{defi}
Notice that only the bispecial factors $F_i$ for $i=0,1,\ldots, d-3$ may have more offsprings since they make part of more bispecial triplets. 

\begin{example}\label{ex:iteraceEpsilonu} In the above notation, $(k,F_k,k) = f^k(0,\varepsilon, 0)$ for $k =0,1,\ldots, d-2$. 
\end{example}
\begin{lemma}\label{lem:vlastnostif} The mapping $f$  has the following properties:

\begin{enumerate}
    \item  $f: \mathcal{T} \mapsto \mathcal{T}$ is injective.

    \item Let $w'$ be a bispecial factor of $\uu$ such that $w'$ has a synchronization point. Then there exists $a,b \in \mathcal{A}, a,b < d-1$  and a bispecial factor $w$ such that $w'$ is the bispecial factor associated with $f(a,w,b)$.   
    \item   Let  $w'$ be the offspring of  a bispecial  factor $w\notin \{F_i:i=0,1,\ldots, d-3\}$. Then  
 $\#\mathrm{Bext}(w) = \#\mathrm{Bext}(w')$.    
\end{enumerate}
    
\end{lemma}
\begin{proof} Item 1  follows  from injectivity of  $\varphi$  and the fact that  $w'=s\varphi(w)p$, where $s$ is the longest common suffix of $\varphi(a) = 0(a+1) =0 a'$ and  $\varphi(d-1) = 0(d-1)(d-1)$ and  analogously, $p$  is the longest common prefix of $\varphi(b) = 0(b+1)=0b'$ and  $\varphi(d-1) = 0(d-1)(d-1)$. 
In other  words, the triplet $(a',w',b')$ determines a~unique triplet $(a,w,b)$ -- the preimage of  $(a',w',b')$ by $f$. 

\medskip

Item 2 is a special case of the result proven in \cite{K2012}.

\medskip
By Proposition~\ref{prop:3extensions} and Example \ref{ex:iteraceEpsilonu}, we have $w' \neq F_i$ for $i =0,1,\ldots, d-2$ and $w'$ has two left extensions, say  $a' $ and $d-1$, and two right extensions, say  $b'$ and $d-1$.  In other words, $w'$ is associated with only one triplet $(a',w',b')\in \mathcal{T}$.  By Item 2 and since $w\not= F_i$ for $i=0,1,\ldots, d-3$, there exists only one triplet $(a,w,b) \in \mathcal{T}$ such that $w'$ is the offspring of $w$.  Obviously,   $\mathrm{Bext}(w') \subset \{a'w'b', a'w'(d-1), (d-1)w'b',  (d-1)w(d-1)\}$.  The form of $f$ implies the following simple observation which proves Item 3:  
$$\begin{array}{rcl}
awb \in {\mathcal L}(\uu)& \Leftrightarrow & a'w'b' \in {\mathcal L}(\uu)\,;\\
aw(d-1) \in {\mathcal L}(\uu)& \Leftrightarrow & a'w'(d-1) \in {\mathcal L}(\uu)\,;\\
(d-1)wb \in {\mathcal L}(\uu)& \Leftrightarrow& (d-1)w'b' \in {\mathcal L}(\uu)\,;\\
(d-1)w(d-1) \in {\mathcal L}(\uu)& \Leftrightarrow & (d-1)w'(d-1) \in {\mathcal L}(\uu)\,.\\
\end{array}
$$  
\end{proof}

To simplify description of all bispecial factors, let us show that the offsprings of a~bispecial factor and of its mirror image are again mirror images of each other. 
\begin{lemma}\label{lem:reversal}
Let $(a,w,b) \in \mathcal T$. If $f(a,w,b)=(a', w', b')$, then $f(b, \overline{w}, a)=(b', \overline{w'}, a')$. 
\end{lemma}
\begin{proof}
If $(a,w,b)$ is a bispecial triplet, then by closedness of ${\mathcal L}(\uu)$ under reversal, $(b,\overline{w}, a)$ is a bispecial triplet, too.
We have $f(a,w,b)=(a',w',b')$ and $f(b, \overline{w}, a)=(b', w'', a')$, where $a', b', w', w''$ are given in Definition~\ref{def:Zobrazeni_f}. 
The  definition also guarantees that $w''=\overline{w'}$ if  $\varphi(\overline{w})0=\overline{\varphi(w)0}.$ This simply follows from the equality $\varphi(k)0=0\overline{\varphi(k)}=\overline{\varphi(k)0}$ for all $k \in \mathcal A$.


\end{proof}
 Lemma~\ref{lem:vlastnostif} and Lemma~\ref{lem:reversal} imply the following corollary.
\begin{corollary}\label{coro:rodinyBispecialu} 
    For $k=0,1, \ldots, d-2$ denote $T_{k} = (0, \varepsilon, k)$ and $T_{d-1} = (0, d-1, 0)$. Then for each bispecial factor $w$ of $\uu$ there exist $k \in \mathcal{A} \text{ and } n\in \mathbb{N} $ such that $w$ or $\overline{w}$ is the bispecial factor associated with $ f^n({T}_k)$.  
\end{corollary}

Now we are ready to prove the main result of this section.
\begin{theorem}\label{thm:ordinaryBS}
All bispecial factors of $\uu_d$ are ordinary. 
\end{theorem}
\begin{proof} Let $w$ be a bispecial factor  of $\uu=\uu_d$. 
By Proposition  \ref{prop:3extensions}, if $w$ has more than two right or left extensions, then $w$ is ordinary. Assume that $w$ has exactly two left and two right extensions. Then $w$ is ordinary if and only if $\#\mathrm{Bext}(w) = 3$.

By Corollary \ref{coro:rodinyBispecialu}, every bispecial factor $w$ is associated with a triplet $f^n(T_i)$ for some $i= 0,1,\ldots, d-1$  and $n \in \N$. We denote $n_i$ the smallest index such that the bispecial factor $w^{(i)}$ associated with $f^{n_i}(T_i)$ does not belong to $\{F_0,F_1, \ldots, F_{d-3}\}$. 

If we show that $\#\mathrm{Bext}(w^{(i)}) = 3$, then by Lemma \ref{lem:vlastnostif} every bispecial factor $w$ associated with  $f^{n}(T_i)$ for  $n\geq n_i$ has $\#\mathrm{Bext}(w) = 3$ and thus $w$ is ordinary.

Consequently, to complete the  proof it is enough to verify that $\#\mathrm{Bext}(w^{(i)}) = 3$ for every $i =0,1,\ldots, d-1$. We will discuss several cases. 

\begin{description}
    \item[{\bf Case} $i=d-1$.]\ \\ 
    In this case $n_{d-1}=0$ as the initial bispecial factor  $d-1\notin \{F_0,F_1, \ldots, F_{d-3}\}$.  
    By inspection of all  factors of $\uu$ of length 3 we can see that 
    
    \smallskip
    \centerline{$\mathrm{Bext}(w^{(d-1)})=\mathrm{Bext}(d-1) = \{0(d-1)(d-1), (d-1)(d-1)0, 0(d-1)0 \}$,} 
    \noindent hence $\#\mathrm{Bext}(w^{(d-1)}) = 3$, as desired. 

 \item[{\bf Case} $i =1,\ldots, d-2$.]\ \\ 
 In this case $n_i =d-1-i$ and $w^{(i)} = F_{d-1-i}(d-1)$. 
 Moreover, $w^{(i)}$ is the offspring of $v=F_{d-2-i}$. By Proposition   \ref{prop:3extensions},  the extensions of $v$ are:  
 $$(d-1)v(d-1), \  (d-2-i)vk \ \text{and} \   kv(d-2-i) \  \text{with} \ k=d-1-i, d-i, \ldots,d-1.$$  
 
 Among images of these extensions, only  $$\varphi\bigl((d-2-i)v(d-2)\bigr), \quad \varphi\bigl((d-2-i)v(d-1)\bigr) \quad \text{and} \quad  \varphi\bigl((d-1)v(d-1)\bigr)$$  
 contain $w^{(i)} = F_{d-1-i}(d-1)$  as its factor. Hence $\#\mathrm{Bext}(w^{(i)}) = 3$, too.  

  \item[{\bf Case} $i =0$.]\ \\   
  In this case $n_0=d-2$ and $w^{(0)}= F_{d-2}$. Moreover,   $w^{(0)}$ is the offspring of  $v= F_{d-3}$.  Using  the list of extensions of $v$ from Proposition \ref{prop:3extensions},  we deduce 

  \centerline{$\mathrm{Bext}(w^{(0)}) = \{(d-2)w^{(0)}(d-1), (d-1)w^{(0)}(d-2), (d-1)w^{(0)}(d-1)\}$}

  \noindent and again $\#\mathrm{Bext}(w^{(0)})=3$. 
\end{description}
    \end{proof}

As the language of $\uu$ is closed under reversal, Theorem~\ref{thm:ordinaryBS} and Theorem~\ref{thm:richness_ordinaryBS} guarantee richness of $\uu$. 
\begin{theorem}\label{coro:richness}
The fixed point $\uu_d$ of the morphism $\varphi_d$ defined in \eqref{eq:morphismFi} is rich.
\end{theorem}

For description of the critical exponent, we need to know the length of bispecial factors. Obviously, it suffices to know their Parikh vectors. 
The definition of the mapping $f$ enables us to describe the relation between the Parikh vector of $w$ and of its offspring. For the vectors of the canonical basis of $\mathbb{R}^d$, we use the notation $e_0, e_1, \ldots, e_{d-1}$. 

\begin{proposition}\label{pro:ParikhBS} 

Let $k \in \mathcal{A}$ be fixed. 
For each $n \in \mathbb{N}$ we denote $w_n$ the bispecial factor associated with $f^n(T_{d-1-k}) $ and $\vec{w}_n$ the Parikh vector of $w_n$. 
Then $\vec{w}_{n+1} = M \vec{w}_n + \ell^{(k)}_n$, where  $M$ is the incidence matrix of the morphism $\varphi$  and  $\ell^{(k)}_n\in \mathbb{R}^d$ is defined:

\begin{description}
    \item[for $k =1,2,\ldots, d-2$]  as \ \ 
$$ \ell^{(k)}_n = \left\{\begin{array}{ll}
e_0+e_{d-1}& \text{ if }\ n =k-1  \!\!\!\mod (d-1)\  \ \text{or}  \ \ \ n=d-2 \!\!\!\mod(d-1)\,;\\
e_0 &\text{  otherwise}\,;
\end{array} \right.
$$
\item[for $k =0$ and $k =d-1$] as \ \ 
 $$ \ell^{(k)}_n = \left\{\begin{array}{ll}
e_0+2e_{d-1}& \text{ if }\ n=d-2 \!\!\!\mod(d-1)\,;\phantom{\  \ \text{or}  \ \ \ n=d-2 \!\!\!\mod(d-1)} \\
e_0 &\text{  otherwise}\,.
\end{array} \right.
$$   
\end{description}
\end{proposition}

\begin{proof} 
The statement is a direct consequence of the form of $f$ from Definition~\ref{def:Zobrazeni_f}.
The bispecial factor $w_{n+1}$ is obtained by applying the morphism $\varphi$ on $w_n$ (Parikh vector $M \vec{w}_n$), appending $0$ (Parikh vector $e_0$) and possibly adding $d-1$ (Parikh vector $e_{d-1}$) to the beginning or/and to the end.
This addition happens with period $d-1$.
    
\end{proof}

\section{Return words to bispecial factors of $\uu_d$ }\label{sec:return}
To compute the asymptotic critical exponent using a result from~\citep{DolceDP2023} (later in this text recalled as Theorem~\ref{thm:FormulaForE}), we need to know the shortest return words to bispecial factors of $\uu_d$. 
The following theorem makes this task easier.  
\begin{theorem}(\citep{BPSteiner2008})\label{thm:praveDreturn} Let  $\uu$ be a $d$-ary uniformly recurrent sequence. If all bispecial factors are ordinary, then each factor has exactly $d$ return words. 
\end{theorem}
The sequence $\uu_d$ is uniformly recurrent and all bispecial factors are ordinary by Theorem~\ref{thm:ordinaryBS}, thus each factor of $\uu_d$ has $d$ return words. Obviously, for every prefix $w$ of $\uu_d$, the sequence $\uu_d$ may be written as a concatenation of return words to $w$. Thus, when coding return words by letters, we obtain a $d$-ary sequence, called a derived sequence. It is interesting to mention that the derived sequence is again equal to $\uu_d$.\footnote{We were pointed out to this by Herman Goulet-Oullet in private communication. More on morphisms preserving return sets in~\citep{Herman}.} 

Again, in the text, we fix $d\geq 3$ and we abbreviate $\varphi = \varphi_d$ and $\uu = \uu_d$. 

\begin{lemma}\label{lem:completeReturn} Let  $w$ be a bispecial factor associated with $(a,w,b) \in \mathcal{T}$ such that $w$ has exactly two left and two right extensions. 

If $r$ is a  return word to $w$, then $s\varphi(r)s^{-1}$  is a return word  to the offspring of $w$, where $s= {\rm lcs\,}\{\varphi(a), \varphi(d-1)\}$. 

In particular, if $\vec{r}$ is the Parikh vector of a return word to $w$, then $M\vec{r}$ is the Parikh vector of a return word to its offspring. 

\end{lemma} 

\begin{proof}    The bispecial factor $w$  has exactly two occurrences in  the complete return word $R=rw$, as a prefix and as a suffix.  There exist $x_1, x_2 \in \{a,d-1\} $ and $y_1,y_2 \in \{b,d-1\}$  such that: 

1) $x_1Ry_1 \in \mathcal{L}(\uu)$ \ \ and \ \  2) $x_1wy_2$ is a prefix  and $x_2wy_1$ is a suffix of  $x_1Ry_1$.

\medskip
 
\noindent By the definition of offspring, $w'=s
  \varphi(w)p $  
   with $s= {\rm lcs\,}\{\varphi(a), \varphi(d-1)\}$  and $p= {\rm lcp\,}\{\varphi(b),\varphi(d-1)\}$. As  $s$ is a suffix of $\varphi(x_1)$ and $\varphi(x_2)$ and $p$ is a prefix of $\varphi(y_1)$ and $\varphi(y_2)$, the factor $w'$ occurs in $\varphi(x_1wy_2)$ and in  $\varphi(x_2wy_1)$. 
  
   The facts that  $\varphi(w)$ has synchronization points $\bullet \varphi(w)\bullet$ and $\varphi$ is an injective morphism guarantee that $\varphi(w)$ occurs in $\varphi(x_1Ry_1)$ only as a factor of $\varphi(x_1wy_2)$ and  $\varphi(x_2wy_1)$.
 In other words, $w'=s\varphi(w)p$ has exactly two occurrences in $\varphi(x_1)\varphi(R) \varphi(y_1)$.    Thus 
 $R'= s \varphi(R)p  = s\varphi(r)s^{-1}s\varphi(w) p$ is a complete return word to $w'$. Consequently, $s\varphi(r)s^{-1} $ is a~return word to $w'$.  
\end{proof}

Proposition~\ref{prop:3extensions} implies that if a bispecial factor has only two left and two right extensions, then its offspring has also the same property. Moreover, by Lemma~\ref{lem:completeReturn}, the return words to offsprings have the same Parikh vectors as images by $\varphi$ of return words to the original bispecial factors. In the sequence of triplets $f^n(T_k)$ we will identify the shortest bispecial factors containing at least one letter $d-1$. We will show that they have only two left and two right extensions and describe their return words. For this purpose, we will use the words $F_k$ and their properties listed in the next lemma.

\begin{lemma}\label{lem:drobnosti} Let $k\in \{0,1,\ldots, d-1\}$ and $F_k$ be the word  defined in  \eqref{eq:DefF}. 
\begin{enumerate} \item $F_k = \varphi^{k-1}(0)\varphi^{k-2}(0)\cdots \varphi(0)0$  \ for  $k=1,2,\ldots, d-1$. 

\item Only letters strictly smaller than $k$ occur in $F_k$. 

\item $F_kk= \varphi^k(0)$.

\item  $\varphi(F_k)0 = \varphi^k(0)F_k$.

\item $\varphi^d(0) = \varphi^{d-1}(0)\varphi^{d-1}(0) (d-1) $.

\item $F_k(d-1)$ is a suffix of $F_{d-1}(d-1) = \varphi^{d-1}(0)$. 

\item $\varphi^k(d-1)F_k(d-1) \in \mathcal{L}(\uu) $.

\item $\varphi^k(d-1) = F_k(d-1)(d-1)\varphi(d-1)\varphi^2(d-1)\cdots \varphi^{k-1}(d-1)$ \ \ for  $k=1,2,\ldots, d-1$.

\item If a letter $i \neq d-1$ occurs in $\varphi^k(d-1)$, then $i<k$.   
\end{enumerate}
\end{lemma}
\begin{proof} Items 1 and 2 follow easily from the definition of $F_k$ and the form of $\varphi$. 

\medskip

Item 3: We proceed by induction. For $k=0$, we see   $F_00 = 0= \varphi^0(0) $. 

\noindent If $k>0$, then by definition of $F_k$ we get $F_kk = \varphi(F_{k-1})0k = \varphi(F_{k-1}(k-1))$. Using the induction hypothesis we deduce $ F_kk= \varphi(\varphi^{k-1}(0)) = \varphi^k(0)$, as needed.

\medskip

Item 4: By  Item 1,  $\varphi(F_k)0 = \varphi^k(0)\varphi^{k-1}(0)\cdots \varphi(0)0=\varphi^k(0) F_{k}$.

\medskip

Item 5:  We apply   Item 3, Item  4 and again Item 3 with $k=d-1$: 
$$
\varphi^d(0) = \varphi(\varphi^{d-1}(0)) = \varphi\bigl(F_{d-1}(d-1)\bigr)= \underbrace{\varphi\bigl(F_{d-1}\bigr) 0}_{\varphi^{d-1}(0)F_{d-1}}(d-1)(d-1) = \varphi^{d-1}(0)\underbrace{F_{d-1}(d-1)}_{\varphi^{d-1}(0)}(d-1)\,. 
$$

Item 6 is a direct consequence of Items 1 and  3. 

\medskip

Item 7: We proceed by induction. If $k=0$, then $(d-1)F_0(d-1) = (d-1)(d-1) \in \mathcal{L}(\uu)$, see the list \eqref{eq:list2factors}. 
Assume that $\varphi^k(d-1)F_k(d-1) \in \mathcal{L}(\uu) $.  
Then $\varphi\Bigl(\varphi^k(d-1)F_k(d-1)\Bigr)$ is a factor of $\uu$, too. As  $\varphi\Bigl(\varphi^k(d-1)F_k(d-1)\Bigr) =   \varphi^{k+1}(d-1)\underbrace{\varphi(F_{k})0}_{F_{k+1}}(d-1)(d-1)$, the statement follows. 

\medskip

Item 8: We proceed again by induction. For $k=1$, the statement holds. Assume that $\varphi^k(d-1) = F_k(d-1)(d-1)\varphi(d-1)\varphi^2(d-1)\cdots \varphi^{k-1}(d-1)$ holds for some $k, \ 1\leq k < d-1$. Then 

$\begin{array}{rcl}
\varphi^{k+1}(d-1)&=&\varphi\Bigl(F_k(d-1)(d-1)\varphi(d-1)\varphi^2(d-1)\cdots \varphi^{k-1}(d-1)\Bigr)\\
&=&\underbrace{\varphi(F_k)0}_{F_{k+1}}(d-1)(d-1)\varphi(d-1)\cdots \varphi^k(d-1)\,.
\end{array}$

\medskip

Item 9: We proceed again by induction. For $k=1$, the statement holds. Assume it holds for some $k, \ 1\leq k < d-1$. Consider letters in $\varphi^{k+1}(d-1) = \varphi^k\bigl(0(d-1)(d-1)\bigr)=\varphi^k(0)\varphi^k(d-1)\varphi^k(d-1)$. By induction assumption, only $d-1$ and letters smaller than $k$ occur in $\varphi^k(d-1)$, and by Item 2, only letters smaller than $k+1$ occur in $\varphi^k(0)$.

\end{proof}

{\bf Notation:} Let $\vec{x} = (x_1,x_2, \ldots, x_d)^T\in \mathbb{R}^{d}$ and $\vec{y} = (y_1,y_2, \ldots, y_d)^T\in \mathbb{R}^{d}$. If $x_i\geq y_i$ for all $i=1,2,\ldots, d$, we write $\vec{x}\geq \vec{y}$.

\medskip

In order to compute the asymptotic critical exponent of $\uu_d$, we need to know the length of the shortest return words to bispecial factors. Obviously, if $r_1$ and $r_2$ are two return words to a~bispecial factor $w$ and if their Parikh vectors satisfy $\vec{r}_1\geq \vec{r}_2$, then $|r_1|\geq |r_2|$. If moreover $w$ meets assumptions of Lemma~\ref{lem:completeReturn}, then non-negativity of the matrix $M$ implies $M\vec{r_1}\geq M\vec{r_2}$, hence the return word $s\varphi(r_1)s^{-1}$ to the offspring of $w$ is longer than or of the same length as the return word $s\varphi(r_2)s^{-1}$.

\begin{proposition}\label{pro:nejkratsiBS}  Let $k \in \{0,1,\ldots, d-1\}$ be fixed. Denote  by  $w$  the bispecial factor associated with the triplet $f^k(T_{d-1-k})$.  
\begin{itemize}
\item If  $k=d-1$,  then
\begin{itemize}
\item $w=(d-1)F_{d-1}(d-1)$ and $w$ has exactly two left and two right extensions;
\item $B:=w\varphi^{d-1}(0)$ is a complete return word to $w$ in $\uu$;
\item $\vec{R}\geq \vec{B}$ for every complete return word $R$ to $w$ in $\uu$.
\end{itemize}

\item If $k< d-1$, then 
\begin{itemize}
    \item $w=F_{k}(d-1)$  and $w$ has exactly two left and two right extensions;
    \item $A:=\varphi^k(d-1)w$  and $B:= w\varphi^{d-1}(0)$ are complete return words to $w$ in $\uu$;
    \item $\vec{R}\geq \vec{A}$ or $\vec{R}\geq \vec{B}$ for every complete return word $R$ to $w$ in $\uu$.  
\end{itemize}

\end{itemize}
\end{proposition}
  
\begin{proof} We discuss separately the cases where the fixed $k$ equals $d-1$ and where $k< d-1$. 

\medskip

{\bf Case} $k=d-1$. \\
\noindent Applying $k$ times the mapping $f$ to the initial triplet $T_{0} =(0,\varepsilon, 0)$, we get the triplet $(0,w,0)$ with $w=(d-1)F_{d-1}(d-1)$. By Proposition~\ref{prop:3extensions}, the bispecial factor $w$ has exactly two left and two right extensions. Clearly, $B$ is a factor of $\uu$ since $(d-1)0$ is a factor and $B$ is a factor of $\varphi^d\bigl((d-1)0\bigr)$ by Items 3 and 5 of Lemma~\ref{lem:drobnosti}. The word $B=w\varphi^{d-1}(0)$ is a complete return word to $w$: by Item 3 of Lemma~\ref{lem:drobnosti}, $w\varphi^{d-1}(0)=(d-1)F_{d-1}(d-1)F_{d-1}(d-1)$, thus $w$ is both a prefix and a suffix of $B$. Moreover, $w$ is not contained in the middle since by Item 2, $F_{d-1}$ does not contain the letter $d-1$. For the same reason, $B$ is the only complete return word to $w$ where the occurrences of $w$ overlap. Hence, every other complete return word $R$ to $w$ satisfies $\vec R \geq \vec B$. \\

{\bf Case} $k\in \{0,1,\ldots, d-2\}$.
\begin{enumerate} 
\item Applying $k$ times, where $1\leq k <d-1$, the mapping $f$ introduced in Definition \ref{def:Zobrazeni_f} to  the initial triplet $T_{d-1-k} =(0,\varepsilon, d-1-k)$, we get the triplet $(k, w, 0)$, where  $w=F_k(d-1)$.  
For $k=0$, we have $w=F_0(d-1)=d-1$, which is the bispecial factor associated with $T_{d-1}=(0,d-1,0)$. By Proposition~\ref{prop:3extensions}, the bispecial factor $w$ has exactly two left and two right extensions.

\item Denote $A_j=\varphi^k(d-1)F_j(d-1)$ for $j=k, k+1, \ldots, d-1$. 
By Item 8 of Lemma~\ref{lem:drobnosti}, the word $w= F_k(d-1)$ is a prefix of $A_j$, by Item 1,  $w$ is a suffix of $A_j$. 

To show that $w$ occurs in $A_j$  just twice, we  exploit the fact that  $w$ contains only one letter $d-1$ and only one letter $k-1$ for $k\geq 1$:   
by Item 2, the letter $d-1$ does not occur in $F_j$ for $j\leq d-1$; by Items 8 and 9 of Lemma~\ref{lem:drobnosti}, the letter $k-1$ for $k\geq 1$ occurs in $\varphi^k(d-1)$ just once. Hence the word   $A_j=\varphi^k(d-1)F_j(d-1)$ contains just two occurrences of $w$.  

To conclude that for every $j=k,k+1, \ldots, d-1$ the word $A_j$ is a complete return to $w$, we need to show that $A_j$ belongs to the language of $\uu$.  
Indeed, 
 the word  $A_{j+1}$ is a factor of $$\varphi(A_j) = \varphi^{k+1}(d-1) \varphi(F_j)\varphi(d-1) = \varphi^k(0)\varphi^k(d-1)\underbrace{\varphi^k(d-1)F_{j+1}(d-1)}_{A_{j+1}}(d-1)$$ for every $j$ with $k\leq j<d-1$. Consequently $A_{j}$ is a factor of $\varphi^{j-k}(A_k)$  and  belongs to the language, as by Item 7, the word  $A_k$ is a factor of $\uu$.

As $F_k$ is a suffix of $F_j$ pro $j\geq k$, the Parikh vectors satisfy $\vec A=\vec{A}_k\leq \vec{A_j}$.

 \item Denote $B_0=B=w\varphi^{d-1}(0)$ and $B_{j} = w(d-1)\varphi(d-1)\cdots \varphi^{j-1}(d-1)\varphi^{d-1}(0)$ for $j=1,2,\ldots, k$. Obviously, $w$ is a prefix of $B_j$ and $\vec{B_j}\geq 
 \vec{B_0}$.   By  Item 6 of Lemma~\ref{lem:drobnosti}, $w=F_{k}(d-1)$ is a suffix of $\varphi^{d-1}(0)$ and thus $w$ is a suffix of $B_j$ as well. 

To deduce that $w$ occurs in $B_j$ just twice, we again use the fact that the letter $d-1$ present in $w$  occurs in $\varphi^{d-1}(0)$ only at the last position, and the letter $k-1$ present in $w$ occurs in $w(d-1)\varphi(d-1)\cdots \varphi^{j-1}(d-1)$ just once. 

Observe that $B_{j+1}$ is a factor of $\varphi(B_j)$. Indeed, by Item 4 of Lemma~\ref{lem:drobnosti},  
$$\varphi(w) = \varphi(F_k)0(d-1)(d-1) = \varphi^k(0)F_k(d-1)(d-1) = \varphi^k(0)w(d-1)$$ and by Item 5, $\varphi^d(0)=\varphi^{d-1}(0) \varphi^{d-1}(0) (d-1)$.  Thus 
$$\varphi(B_j) = \varphi^k(0)\underbrace{w(d-1)\varphi(d-1)\cdots \varphi^j(d-1)\varphi^{d-1}(0)}_{B_{j+1}}\varphi^{d-1}(0)(d-1).  $$
Since $B_0$ belongs to the language by Items 6 and 5 of Lemma~\ref{lem:drobnosti}, all factors of $\varphi^j(B_0)$, including $B_j$, belong to ${\mathcal L}(\uu)$.  
 \end{enumerate}

\bigskip
To summarize, 
we have shown that  $B=B_0, B_1, \ldots, B_{k}$  and $A=A_k, A_{k+1}, \ldots, A_{d-1}$ are  complete return words to $w$ in $\uu$. Clearly, $\vec{B}=\vec{B}_0\leq \vec{B_j}$ for every $ j=0,1,\ldots, k$ and $\vec A=\vec{A}_k\leq \vec{A_j}$ for every $ j=k, k+1, \ldots, d-1$.

By Items 3 and  8 of Lemma~\ref{lem:drobnosti}, $B_{k} = A_{d-1}= \varphi^k(d-1)\varphi^{d-1}(0)$. But other items in our list do not coincide. 
Indeed: for $j< d-1$, the penultimate occurrence of the letter $d-1$ in $A_j$ is followed by the suffix of length $|F_j(d-1)| < |F_{d-1}(d-1)|=|\varphi^{d-1}(0)|$, whereas the penultimate occurrence of $d-1$ in each $B_j$ is followed by the suffix of length  $|\varphi^{d-1}(0)|$. 

Hence, we have found $d$  distinct complete return words to $w$ in  $\uu$. By Theorem~\ref{thm:praveDreturn},  every complete return word $R$ to $w$ in $\uu$ occurs in the list $B_0, B_1, \ldots, B_{k}, A_k, A_{k+1}, \ldots, A_{d-1} = B_k$.  

\end{proof}
\begin{remark}\label{rem:JasneKratisReturn}  The above proposition does not decide for a general fixed $k\in \{0,1,\ldots, d-2\}$, whether $A$ or $B$  is a shorter complete return word to $w$, the bispecial factor associated with the triplet $f^k(T_{d-1-k})$. For $k=0$, clearly, $\vec A \leq \vec B$, thus $A$ is the shortest return word to $w$. 
    
\end{remark}

%


\section{The asymptotic critical exponent of $\uu_d$ }\label{sec:asymptotic}

We have already all ingredients needed to compute the asymptotic critical exponent of $\uu_d$, for $d\geq 3$. To do so, we make use of a result taken from~\citep{DolceDP2023}.  

\begin{theorem}(\citep{DolceDP2023})\label{thm:FormulaForE}
Let $\uu$ be a uniformly recurrent aperiodic sequence.
Let $(w_n)_{n\in\N}$ be the sequence of all bispecial factors in $\uu$ ordered by length.
For every $n \in \N$, let $r_n$ be the shortest return word to the bispecial factor $w_n$ in $\uu$.
Then
$$
E^*(\uu) = 1 + \limsup\limits_{n \to \infty}  \frac{|w_n|}{|r_n|} .
$$
\end{theorem}
To compute the limits relevant for the previous formula, we will need the following technical lemma on primitive matrices, its proof may be found in Appendix.
\begin{lemma}\label{lem:VypocetLimity}
Let $M \in \mathbb{Z}^{d\times d}$ be a non-negative primitive matrix. Assume that sequences  $(x_n)_{n \in \N}$,  $(y_n)_{n \in \N}$ and $(\ell_n)_{n \in \N}$ of non-negative column vectors from   $ \mathbb{Z}^{d}$ have the following properties: 
\begin{enumerate}
\item  $ x_{n+1} = Mx_n + \ell_n$  and  $y_{n+1} =M{y_n} $  for every  $n \in \N$;       \item $(\ell_n)_{n \in \N}$ is purely periodic with  period  $p\in \N$.   
\end{enumerate}
 We denote by  $z\in \mathbb{R}^d$ a left eigenvector of $M$ to its spectral radius $\Lambda$ and by $e$  the row vector $(1,1,\ldots, 1)\in \mathbb{R}^d$. Then  
$$
\lim_{n\to +\infty} \frac{~~ex_n~~}{ey_n} = \frac{zc+ (\Lambda^p-1)zx_0}{(\Lambda^p-1)zy_0}\,,
$$
where $c = M^{p-1}\ell_0 + M^{p-2}\ell_1 +\cdots + M\ell_{p-2} +  \ell_{p-1}$. 
\end{lemma}

We will apply Lemma~\ref{lem:VypocetLimity} to recurrence relations provided in Proposition~\ref{pro:ParikhBS}. The matrix $M$ is thus the incidence matrix of the morphism $\varphi=\varphi_d$, i.e., 
$$M=M_{\varphi_d}= 
\left(\begin{smallmatrix}
1&1&\cdots &1&1&1\\
1&0&\cdots &0&0&0\\
0&1&\cdots &0&0&0\\
\vdots&&\ddots&\ddots&&\vdots\\
&&&&&\\
0&0&\cdots &1&0&0\\
0&0&\cdots &0&1&2\\
\end{smallmatrix}\right)\in \mathbb{R}^{d\times d}\,.$$
It is easy to see that the characteristic polynomial of $M$ equals 
\begin{equation}\label{eq:polynomProLmabda}
\chi_d(t) = t^d-3t^{d-1}+t^{d-2} + t^{d-3} + \cdots + t^2+ t + 1 = \tfrac{1}{t-1}\, \Bigl({t^{d-1}(t-2)^2 -1}\Bigr)    
\end{equation}  
and a left  eigenvector $z \in \mathbb{R}^d$ of $M$ to the dominant eigenvalue (spectral radius) $\Lambda$ is 

\begin{equation}\label{eq:eigenvector}
z =\Bigl((\Lambda -2)\,, (\Lambda-2)s_1\,,   (\Lambda-2)s_2\,, \ldots\,, (\Lambda-2)s_{d-2}\,,1\Bigr)\,,  
\end{equation}
where $s_i = \Lambda^{i} - \sum_{k=0}^{i-1}\Lambda^k$.

\begin{proposition}\label{pro:LimityTriplety} Let $k\in \{0,1,\ldots, d-1\}$ be fixed. For every $n \in \mathbb{N}$,  denote by $w_n$ the bispecial factor associated with the triplet $f^n(T_{d-1-k})$ and by $r_n$ the shortest return word to $w_n$. Then 
$$\limsup_{n\to \infty}\frac{|w_n|}{|r_n|} \leq \frac{1}{3-\Lambda}.$$
The equality is attained for $k=0$ and $k=d-1$. 
\end{proposition}

\begin{proof} Let us recall that the length $|u|$ of a word $u \in \mathcal{A}^*$ equals  $e\vec{u}$, where $e=(1,1,\ldots, 1)$.  By Proposition \ref{pro:ParikhBS},  the Parikh vectors $\vec{w}_n$ meet the recurrence relation $\vec{w}_{n+1} = M\vec{w}_n +\ell^{(k)}_n$, with the starting condition $\vec{w}_0 = \vec{0}$ for $k=1,2,\ldots, d-1$  and $\vec{w}_0 = e_{d-1}$ for $k=0$.  

Define recursively two sequences of vectors: 

\begin{itemize}
    \item 
$\vec{t}_0 = e_{d-1}$  \ \ and  \ \ $\vec{t}_{n+1} = M\vec{t}_n$  \ \ for every $n \in\mathbb{N}$;

\item $\vec{s}_0 = M^{d-1-k}e_0$  \ \   and  \ \ $\vec{s}_{n+1} = M\vec{s}_n$  \ \ for every $n \in\mathbb{N}$. 
\end{itemize}

\noindent For each $n\geq k$ we will show that the Parikh vector of the shortest return word to $w_n$ in $\uu$ is either $\vec{t}_n$ or $\vec{s}_n$.  

Indeed: by Proposition \ref{pro:nejkratsiBS}, either $\varphi^k(d-1)w_k$ or $w_k\varphi^{d-1}(0)$ is the shortest complete return word to $w_k$, thus the Parikh vector of the shortest return word to $w_k$ equals either to the Parikh vector of  $\varphi^k(d-1)$  or $\varphi^{d-1}(0)$, i.e., to $M^{k}e_{d-1}= \vec{t}_k $ or to $M^{d-1}e_0=\vec{s}_k$.

By Proposition \ref{pro:nejkratsiBS}, the bispecial factor $w_k$ has only two left and two right extensions. 
Hence by Lemma \ref{lem:completeReturn}, for $n\geq k$, the Parikh vector of the shortest return word to $w_{n}$  is either $\vec{t}_n$ or $\vec{s}_n$.

\medskip

Since $$\frac{|w_n|}{|r_n|} = \frac{e\vec{w}_n}{e\vec{r}_n}\leq \max\Bigl\{\frac{e\vec{w}_n}{e\vec{s}_n}, \frac{e\vec{w}_n}{e\vec{t}_n} \Bigr\}\,,$$
to prove the theorem, it suffices to show that the limits $L_k:=\lim\limits_{n\to \infty}\frac{e\vec{w}_n}{e\vec{t}_n}$ and $\tilde{L}_k:=\lim\limits_{n\to \infty}\frac{e\vec{w}_n}{e\vec{s}_n}$ are less than or equal to $\frac{1}{3-\Lambda}$. 

To compute the limits, we use Lemma~\ref{lem:VypocetLimity}, where $x_n = \vec{w}_n$,  $y_n = \vec{s}_n$, resp. $y_n = \vec{t}_n$, and the sequence $\ell^{(k)}_n$ described in Proposition~\ref{pro:ParikhBS} is purely periodic with period $d-1$. The left eigenvector of the matrix $M$ corresponding to $\Lambda$ is provided in~\eqref{eq:eigenvector}, hence 
\begin{equation}\label{eq:Jmenovatele} \ z\vec{t}_0 = 1 \  \   \text{and} \ \  z\vec{s}_0 = \Lambda^{d-1-k}(\Lambda -2).\end{equation} 

\begin{description}
   \item[Case $k= 0$.] \quad Consider the triplets $f^n(T_{d-1})$. From the form of the sequence $\ell^{(0)}_n$ described in Proposition \ref{pro:ParikhBS}, we determine the vector $c =  \bigl(\sum_{j=0}^{d-2}M^j\bigr)e_0 + 2e_{d-1}$ and thus
  
    \centerline{$zc =  \Bigl(\, \sum\limits_{j=0}^{d-2} \Lambda^j\Bigr) \underbrace{ze_0}_{\Lambda - 2} +2\underbrace{ze_{d-1}}_{1} = \tfrac{\Lambda^{d-1} - 1}{\Lambda -1}(\Lambda -2) + 2$.}
By Lemma \ref{lem:VypocetLimity}
  $$
    L_0 =\frac{\tfrac{\Lambda^{d-1} - 1}{\Lambda -1}(\Lambda -2) + 2 + (\Lambda^{d-1}-1)}{\Lambda^{d-1}-1}  =
    \frac{2\Lambda - 3}{\Lambda -1}+ \frac{2}{\Lambda^{d-1}-1} = \frac{1}{3-\Lambda}\,.  
   $$
The last equality follows from the fact that $\Lambda$ is a root of the polynomial $t^{d-1}(t-2)^2-1$, see \eqref{eq:polynomProLmabda}, therefore $\frac{1}{\Lambda^{d-1} - 1 }= \tfrac{(\Lambda-2)^2}{(\Lambda -1)(3-\Lambda)}$. 

By Remark \ref{rem:JasneKratisReturn}, it is clear that $\tilde{L}_0 \leq L_0$, thus there is no need to compute it.

 \item[Case $k= d-1$.] \quad Consider the triplets $f^n(T_0)$. The sequence $\ell^{(d-1)}_n$ is the same as $\ell^{(0)}_n$, therefore $zc$ is the same as for $k=0$.
By the form of the shortest return word, we only need to compute 
$$
    \tilde{L}_{d-1}= \frac{\tfrac{\Lambda^{d-1} - 1}{\Lambda -1}(\Lambda -2) + 2}{(\Lambda^{d-1}-1)(\Lambda -2)} = \frac{1}{3-\Lambda}\,.  
   $$

    \item[Case $k\in \{1,\dots, d-2\}$.] \quad Consider the triplets $f^n(T_{d-1-k})$. From the form of the sequence $\ell^{(k)}_n$ described in Proposition \ref{pro:ParikhBS}, we determine the vector $c =  \bigl(\sum_{j=0}^{d-2}M^j\bigr)e_0 + \bigl(M^{d-1-k}+I\bigr)e_{d-1}$. Thus

    \centerline{$zc =  \Bigl(\, \sum\limits_{j=0}^{d-2} \Lambda^j\Bigr) \underbrace{ze_0}_{\Lambda - 2} + \bigl(\Lambda^{d-1-k} +1)\underbrace{ze_{d-1}}_{1} = \tfrac{\Lambda^{d-1} - 1}{\Lambda -1}(\Lambda -2) + \Lambda^{d-1-k} +1$.}

\medskip
 \noindent By Lemma \ref{lem:VypocetLimity}
 
    $$
    L_k =\frac{\tfrac{\Lambda^{d-1} - 1}{\Lambda -1}(\Lambda -2) + \Lambda^{d-1-k} +1}{\Lambda^{d-1}-1}  < L_0 = \frac{1}{3-\Lambda}  
   $$
   and 
$$
   {\tilde L}_k = \frac{\tfrac{\Lambda^{d-1} - 1}{\Lambda -1}(\Lambda -2) + \Lambda^{d-1-k} +1}{(\Lambda^{d-1}-1)(\Lambda -2)\Lambda^{d-1-k}} < \tilde{L}_{d-1}= \frac{1}{3-\Lambda}. 
   $$
 \end{description}
\end{proof}

\begin{theorem}\label{thm:HodnotaE*}  Let  $\uu_d$ be the fixed point of the morphism $\varphi_d$ defined in \eqref{eq:morphismFi}.
Then the asymptotic critical exponent of $\uu_d$ equals 
$$E^*(\uu_d) = 1+\frac{1}{3-\Lambda_d},$$
 where $\Lambda_d$  is the spectral radius of the incidence matrix of the morphism $\varphi_d$. 
\end{theorem}
\begin{proof} The statement follows immediately from Corollary~\ref{coro:rodinyBispecialu} and Theorem~\ref{thm:FormulaForE} and from Proposition~\ref{pro:LimityTriplety}. 
    
\end{proof}

\section{The asymptotic repetition threshold of sequences rich in palindromes\\ (proof of Theorem \ref{thm:main})}\label{sec:konecne}
The main aim of this section is to prove Theorem~\ref{thm:main}. For the proof, we need to show that the image of $\uu=\uu_d$ under the projection  $\pi: \{ 0,1,\ldots, d-1\}^* \mapsto \{  \tt a,b\}^*$ defined by
\begin{equation}\label{def:pi}
\pi: { i} \mapsto {\tt a b}^{ i} \qquad \text{for every }\  { i} \in \{ 0,1,\ldots, d-1\}
\end{equation}
 is a rich sequence with the same asymptotic critical exponent as $\uu$.

The following properties of the morphisms $\varphi$ and $\pi$ will be used in the sequel:
\begin{itemize}
\item Let  $w\in \{0,1,\dots, d-1\}^*$. Then 
  \begin{itemize}
        \item $w$ is a palindrome if and only if $\varphi(w)0$ is a palindrome; 
\item $w$ is a palindrome if and only if $\pi(w){\tt a}$ is a palindrome.
    \end{itemize} 
    Consequently, ${\mathcal L}(\pi(\uu))$ contains infinitely many palindromes.
\item The sequence $\pi(\uu)$ is uniformly recurrent since $\uu$ is uniformly recurrent.
\item ${\mathcal L}(\pi(\uu))$ is closed under reversal. This follows putting together the two previous observations.   
\end{itemize}

As the first step, we will show that the projection is rich.
\begin{theorem}\label{thm:RichnessProjection} Let  $\uu_d$ be the fixed point of the morphism $\varphi_d$ defined in \eqref{eq:morphismFi} and let $\pi$ be the projection defined in~\eqref{def:pi}.  Then $\pi(\uu_d)$ is rich.        
\end{theorem} 

This time, we cannot use Theorem~\ref{thm:richness_ordinaryBS} since there are bispecial factors of $\pi(\uu_d)$ that are not ordinary. For instance, ${\rm{Bext}}(\varepsilon)=\{\tt aa, ab, ba, bb\}$, hence $\mathrm{b}(\varepsilon)=1$. The main ingredient for the proof of Theorem~\ref{thm:RichnessProjection} will be now a theorem relating richness and bilateral order of bispecial factors. 
\begin{theorem}[Theorem 3.10~\cite{BaPeSt2010}]\label{thm:bilateralorder}
Let $\uu$ be a sequence with language closed under reversal.
Then $\uu$ is rich if and only if each bispecial factor $w$ of $\uu$ satiesfies:
\begin{itemize}
\item if $w$ is not a palindrome, then $\mathrm{b}(w)=0$;
\item if $w$ is a palindrome, then $\mathrm{b}(w)=\#\mathrm{Pext}(w)-1$.
\end{itemize}
\end{theorem}

\begin{lemma}\label{lem:ordinary}
Let $w$ be a bispecial factor of $\pi(\uu_d)$ containing the letter $\tt a$, then $w$ is ordinary.
Moreover, if $w$ is a palindrome, then $\#{\mathrm{Pext}}(w)=1$.
\end{lemma}

\begin{proof}
Since $w$ contains the letter $\tt a$, there exists a factor $v\in {\mathcal L}(\uu_d)$ (possibly empty) such that
$$w={\tt b}^i \pi(v){\tt ab}^j\,, \quad \text{where $i, j \in \mathbb N, \ i <d-1, \ j<d-1$.}$$
By closedness under reversal, we may assume without loss of generality that $i\leq j.$
Obviously, the factor $v$ is bispecial.
\begin{description}
\item[Case $1$.] Assume $v$ has only two left and two right extensions. Then by the fact that $v$ is ordinary, we conclude immediately that $w$ is ordinary, too.
If moreover $w$ is a~palindrome, then $i=j$ and $v$ is a palindrome, too. Since $v$ is ordinary and the language is closed under reversal, either ${\mathrm{Bext}}(v)=\{iv(d-1), (d-1)vi, (d-1)v(d-1)\}$ or ${\mathrm{Bext}}(v)=\{iv(d-1), (d-1)vi, ivi\}$. 
Consequently, 
$\mathrm{Bext}(w)=\{{\tt a}w{\tt b}, {\tt b}w{\tt a}, {\tt b}w{\tt b}\}$ or $\mathrm{Bext}(w)=\{{\tt a}w{\tt b}, {\tt b}w{\tt a}, {\tt a}w{\tt a}\}$, hence $\#\mathrm{Pext}(w)=1$.
\item[Case $2$.] Assume $v$ has more than two left or two right extensions. Then by Proposition~\ref{prop:3extensions}, $v=F_i$ and $i\leq d-3$. Hence
$$\mathrm{Bext}(v)=\left\{ivk \ : \ i<k\leq  d-1\right\} \cup \left\{kvi \ : \ i<k\leq d-1\right \}\cup \{(d-1)v(d-1)\}\,.$$
For $i=j$, i.e., $w={\tt b}^i \pi(v){\tt ab}^i$, projection of factors $ivk, kv i, (d-1)v(d-1)$  gives us $$\mathrm{Bext}(w)=\{{\tt a}w{\tt b}, {\tt b}w{\tt a}, {\tt b}w{\tt b}\},$$ thus $w$ is ordinary and $w$ is a palindrome with $\#\mathrm{Pext}(w)=1$.\\
For $i<j$, only projection of factors $iv j, (d-1)v(d-1), ivk$ for $k>j$ contains $w$. It leads to $$\mathrm{Bext}(w)=\{{\tt a}w{\tt a}, {\tt b}w{\tt b}, {\tt a}w{\tt b}\},$$ thus $w$ is ordinary, too.
\end{description}
\end{proof}

\begin{proof}[Proof of Theorem~\ref{thm:RichnessProjection}]

Since the language of $\pi(\uu)$ is closed under reversal, we may use Theorem~\ref{thm:bilateralorder}.
Let $w$ be a bispecial factor of $\pi(\uu)$. There are three possible cases:
\begin{enumerate}
\item $w$ is non-palindromic: Then $w$ contains the letter $\tt a$ and by Lemma~\ref{lem:ordinary} $w$ is ordinary, i.e., $\mathrm{b}(w)=0$. 
\item $w$ is palindromic and $w$ contains the letter $\tt a$: Then by Lemma~\ref{lem:ordinary} $w$ is ordinary and $\#\mathrm{Pext}(w)=1$, hence $\#\mathrm{Pext}(w)-1=0=\mathrm{b}(w)$.
\item $w$ is palindromic and does not contain $\tt a$: Then $w=\varepsilon$, resp. $w={\tt b}^i$ with $i<d-2$, resp. $w={\tt b}^{d-2}$. The reader may easily check that $\mathrm{Bext}(\varepsilon)=\{\tt aa, ab, ba, bb\}$, resp. $\mathrm{Bext}(w)=\{{\tt a}w{\tt a}, {\tt a}w{\tt b}, {\tt b}w{\tt a}, {\tt b}w{\tt b}\}$, resp. $\mathrm{Bext}(w)=\{{\tt a}w{\tt a}, {\tt a}w{\tt b}, {\tt b}w{\tt a}\}$. In all cases, $\#\mathrm{Pext}(w)-1=1=\mathrm{b}(w)$.  
\end{enumerate}

\end{proof}

As the second step, we will show that the projection does not change the asymptotic critical exponent. The following theorem is a crucial tool for that.
\begin{theorem}[\cite{Ochem}]\label{thm:E*_morphic_image} Let ${\bf v}$ be a sequence over an alphabet $\mathcal A$ such that the uniform letter frequencies in ${\bf v}$ exist. Let $\psi:{\mathcal A}^* \to {\mathcal B}^*$ be an injective morphism and let $L \in \mathbb N$ be such that every factor $x$ of $\psi({\bf v})$, $|x|\geq L$, has a~synchronization point. 
Then $E^*({{\bf v}})=E^*(\psi({{\bf v}}))$.
\end{theorem}

\begin{corollary}\label{coro:E*pi}
Let  $\uu_d$ be the fixed point of the morphism $\varphi_d$ defined in \eqref{eq:morphismFi} and let $\pi$ be the projection defined in~\eqref{def:pi}.  
Then $E^*(\pi(\uu_d))=E^*(\uu_d)$.
\end{corollary}
\begin{proof}
Since $\varphi_d$ is a primitive morphism, the uniform letter frequencies in its fixed point $\uu_d$ exist. Moreover, $\pi: \{ 0,1,\ldots, d-1\}^* \mapsto \{\tt a,b\}^*$ is an injective morphism and each factor of $\pi(\uu_d)$ containing the letter $\tt a$ has a synchronization point, hence each factor of length $\geq d$ has a synchronization point. We may therefore apply Theorem~\ref{thm:E*_morphic_image} and obtain $E^*(\pi(\uu_d))=E^*(\uu_d)$.
\end{proof}

As the last step, we will prove our  main result -- Theorem~\ref{thm:main}  on the repetition threshold of rich recurrent sequences.
The following theorem and lemma will be used in the proof.

\begin{theorem}(\citep{PStarosta2013})\label{thm:overlap} Let $\uu$ be a rich recurrent sequence.  Then $\uu$ contains infinitely many
overlapping factors, i.e., the set
$\{www' \in \mathcal{L}(\uu) :  w' \text{\ is a non-empty prefix of} \ w\}$ is infinite.   
\end{theorem}

\begin{lemma}\label{lem:lim_lambda}  Let $(\Lambda_d)_{d \geq 3}$ be the sequence of spectral radii of the incidence matrices of $\varphi_d$. Then $\Lambda_d\in (2,3)$ and $\lim\limits_{d \to \infty}\Lambda_d = 2$. 
\end{lemma}
\begin{proof} The characteristic polynomial $\chi_d$ is described in~\eqref{eq:polynomProLmabda}. By a simple examination of the equation $t^{d-1}(t-2)^2-1=0$, we deduce that the largest real root belongs to $(2,3)$. Since $\Lambda_d\in (2,3)$, we have $(\Lambda_d-2)^2=\frac{1}{{\Lambda_d}^{d-1}}\to 0$ with $d \to \infty$, hence $\Lambda_d$ tends to $2$. 
\end{proof}

\begin{proof}[Proof of Theorem~\ref{thm:main}] 
By Theorem \ref{thm:overlap}, $\RT^*(C^{(r)}_d)\geq 2$ for every $d \geq 2$.

Let us observe that $\RT^*(C^{(r)}_d) \geq \RT^*(C^{(r)}_{d+1})$. Indeed, if we apply the morphism $\sigma: \{0,1,\dots, d-1\}^*\to \{0,1,\dots, d\}^*$, defined by $\sigma(i)=id$ for each $i\in \{0,1,\dots, d-1\}$, to any rich recurrent sequence over $\{0,1,\dots, d-1\}$, the obtained $(d+1)$-ary sequence is again rich, recurrent and has the same asymptotic critical exponent. 

To complete the proof, it thus suffices to show that $\RT^*(C^{(r)}_2) \leq 2$.   
Since $\pi(\uu_d)$ belongs to the class $C^{(r)}_2$ for every $d\geq 3$, Corollary \ref{coro:E*pi} implies  $$\RT^*(C^{(r)}_2) = \inf\{E^*(\uu): \uu \ \text{\ belongs to the class\ } C^{(r)}_2\} \leq E^*(\pi(\uu_d))=   1+\frac{1}{3-\Lambda_d}$$
and by Lemma \ref{lem:lim_lambda}, $\RT^*(C^{(r)}_2) \leq 2.$ 
\end{proof} 

\section{Comments and open problems}
\begin{itemize}
\item In this paper, we have shown that the asymptotic repetition threshold  $\RT^*(C^{(r)}_d) =2$  for the class of rich recurrent sequences over alphabet of size $d\geq 2$.  It remains an open problem to determine the repetition threshold $\RT(C^{(r)}_d)$, where $d\geq 4$.
\item  In our proof, the fixed point $\uu_d$ of the morphism $\varphi_d$ played a crucial role.  The following relation of $\uu_d$ to the repetition threshold   seems to be interesting: $\RT(C^{(r)}_2) = E^*(\uu_3)$  and  $\RT(C^{(r)}_3) = E^*(\uu_5)$. May a similar equality apply to larger values $d$, too? 

\item As mentioned in Introduction for the class  $C^{(e)}_d$ of $d$-ary episturmian sequences,  the values of the asymptotic repetition threshold and the repetition threshold coincide. The class  $C^{(e)}_d$ is $S$-adic. More precisely, for each $d$ there exists $d$ morphisms over a~$d$-letter alphabet $\sigma_0, \sigma_1, \ldots, \sigma_{d-1}$ such that every standard $d$-ary episturmian sequence $\uu$ can be written as  $\uu=\lim\limits_{n\to \infty} \sigma_{i_0}\circ\sigma_{i_1} \circ \cdots \circ\sigma_{i_n}(0)$  for some directive sequence $(i_n)_{n \in \N}$. 

As shown  in~\citep{BaPeSt2011}, any rich sequence $\uu$ can be written in the same form with  the morphisms $\sigma_i$ belonging to the  class $P_{\mathrm{ret}}$.  This class contains morphisms over alphabets of any size. It seems that for finding the $S$-adic representation of a $d$-ary rich sequence, we cannot limit ourselves to  morphisms of $P_{\mathrm{ret}}$  with a bounded size of alphabet.

\item Sequences coding the exchange of $d$ intervals form  an $S$-adic system, too~\cite{FerZam2008}. The corresponding morphisms are over an alphabet of size $d$. To determine the repetition threshold and the asymptotic repetition threshold for this class of sequences is a natural step towards understanding the relationship between the repetition threshold and the asymptotic repetition threshold for $S$-adic systems.

 Moreover,  if the exchange of $d$ intervals is governed by a symmetric permutation, then any sequence coding this exchange is rich. Determination of the repetition threshold for this class of sequences would improve the known upper bound on  $\RT(C^{(r)}_d)$.

\item Morphisms preserving richness have not been completely described yet. On one hand, a class of such morphisms was given in~\citep{GlJuWi2009}. Morphisms of class $P_{\mathrm{ret}}$ which preserve richness over a~binary alphabet were characterized in~\citep{DolceP2022}. 

\item  If a sequence $\uu$ is rich, then its language contains infinitely many distinct  palindromes. Hof, Knill, and Simon in 1995~\citep{HKS1995} defined  Class $P$ of primitive morphisms and conjectured that  the language of a fixed point $\uu$ of a primitive morphism has infinitely many palindromes if and only if $\mathcal{L}(\uu)$ coincides with the language of a~fixed point  of a morphism  from Class $P$. This conjecture, almost 30 years old,  has so far been confirmed only for binary alphabets \cite{BoTan2007}.  It is worth noting that the  class $P_{\mathrm{ret}}$  belongs to  Class $P$. 

\end{itemize}

\section{Appendix (proof of Lemma \ref{lem:VypocetLimity})}
\begin{proof} Let us first concentrate on the computation of the limit of the subsequence ${ex_{pn}}/{ey_{pn}}$. 
Item~1. implies $x_{(n+1)p} = M^px_{np} + c$. Consequently,  
\begin{equation}\label{eq:limita}
x_{np} = M^{pn}x_0 +\Bigl(\sum_{k=0}^{n-1}{M^{kp}}\Bigr)c.
\end{equation}
As the matrix $M$ is primitive, there exists a regular matrix $R$ such that $R^{-1}MR = \left(\begin{array}{c|c}\Lambda& 0\\ \hline0& J\end{array}\right)$, where $J$ is a $(d-1)\times (d-1)$ matrix with the spectral radius strictly smaller than $\Lambda$. Note that the first row of the matrix $R^{-1}$ is a left eigenvector of $M$ to $\Lambda$. By primitivity of $M$ each such eigenvector is a multiple of $z$ and we can without loss of generality assume in the sequel it is equal to $z$. Hence
\begin{equation}\label{eq:limitaMaticova}\lim_{n\to \infty} \Lambda^{-n} eM^n = \lim_{n\to \infty}eR\left(\begin{array}{c|c}
    1 & 0 \\
\hline 0  & \Lambda^{-n} J^n    
\end{array}\right)R^{-1} = eR\left(\begin{array}{c|c}
    1 & 0 \\
\hline 0 &  0
\end{array}\right)R^{-1}  =\gamma z, 
\end{equation}
where $\gamma = R_{11}+ R_{21}+\cdots + R_{d1}$. 
To compute the limit of $ex_{np}/ey_{np}$, we use the Stolz-C\`{e}saro theorem. Using \eqref{eq:limita} we have 
$$ y_{(n+1)p} - y_{np} = M^{pn}(M^p - I)y_0
\qquad \text{and}\qquad x_{(n+1)p} - x_{np} = M^{pn}\Bigl((M^p - I)x_0+c\Bigr).
$$
By \eqref{eq:limitaMaticova}
$$
\lim_{n\to \infty}\Lambda^{-np}e\bigl(y_{(n+1)p} - y_{np}\bigr) = \lim_{n\to \infty}\Lambda^{-np}e M^{pn}(M^p - I)y_0 = \gamma z (M^p - I)y_0  = \gamma(\Lambda^p - 1)zy_0 
 $$
and 
$$
\lim_{n\to \infty}\Lambda^{-np}e\bigl(x_{(n+1)p} - x_{np}\bigr) = \lim_{n\to \infty}\Lambda^{-np}e M^{pn}\Bigl((M^p - I)x_0+c\Bigr) = \gamma\bigl( (\Lambda^p - 1)zx_0 + zc\bigr).
 $$
The Stolz-C\`{e}saro theorem gives 
$$
\lim_{n\to +\infty} \frac{~~ex_{pn}~~}{ey_{pn}} = \lim_{n\to +\infty} \frac{e\bigl(x_{(n+1)p} - x_{np}\bigr)}{~~{e\bigl(y_{(n+1)p} - y_{np}\bigr)}~~}=\frac{zc+ (\Lambda^p-1)zx_0}{(\Lambda^p-1)zy_0} = :L.
$$
We show that $\lim\limits_{n\to +\infty} \frac{~~ex_{pn+1}~~}{ey_{pn+1}} $ equals $L$, too. To show it  we use the previous result applied to the sequences $(x'_n) = (x_{n+1})$, $(y'_{n}) = (y_{n+1})$  and $(\ell'_n) = (\ell_{n+1})$ and thus
$$
\lim_{n\to +\infty} \frac{~~ex_{pn+1}~~}{ey_{pn+1}} = \frac{zc'+ (\Lambda^p-1)zx'_0}{(\Lambda^p-1)zy'_0} = : L'. 
$$
 Now $c' = M^{p-1}\ell'_0 + M^{p-2}\ell'_1+ \cdots + M\ell'_{p-2} + \ell'_{p-1} =  Mc - (M^p-I)\ell_0$. Replacing  $x'_0$ by $ x_1 = Mx_0+ \ell_0$ and $y'_0$ by $My_0$ gives the denominator of the previous fraction $(\Lambda^p-1)zy'_0 =(\Lambda^p-1)zMy_0 = \Lambda(\Lambda^p-1)zy_0$ and the numerator   
$$
zc'+ (\Lambda^p-1)zx'_0 = z\Bigl( Mc - (M^p-I)\ell_0\Bigr)+ (\Lambda^p-1)z(Mx_0+\ell_0)= \Lambda\Bigl( zc +(\Lambda^p-1)zx_0 \Bigr).
$$
Consequently, $L'= L$. 

By the same reasoning  $\lim\limits_{n\to +\infty} \frac{~ex_{pn+i}~}{ey_{pn+i}} $ equals $L$ for every $i =1,2,\ldots, p-1$. 
\end{proof}
\end{document}